\documentclass[12pt, reqno]{amsart}
\usepackage{amsmath, amsthm, amscd, amsfonts, amssymb, graphicx, color, booktabs, multirow, rotating}
\usepackage{enumitem}
\newtheorem{theorem}{Theorem}[section]
\newtheorem{lemma}[theorem]{Lemma}
\newtheorem{proposition}[theorem]{Proposition}
\newtheorem{corollary}[theorem]{Corollary}
\theoremstyle{definition}

\newtheorem{example}[theorem]{Example}

\theoremstyle{remark}
\newtheorem{remark}[theorem]{Remark}
\numberwithin{equation}{section}

\usepackage{listings}

\title{HURWITZ SERIES RINGS SATISFYING A ZERO DIVISOR
PROPERTY}

\author{Behrooz Mosallaei, Moein Afrouzmehr, Danial Abshari and Sepideh Farivar}
\address{The Klipsch School of Electrical and Computer Engineering, New Mexico State University,
Las Cruces, NM 88003 USA}
\email{behrooz@nmsu.edu}
\address{Department of Economics, Applied Statistics, and International Business, New Mexico State University,
Las Cruces, NM 88003 USA}
\email{moeinafr@nmsu.edu}
\address{Software and Information Systems Department,
University of North Carolina at Charlotte, Charlotte, NC 28223 USA}
\email{dabshari@uncc.edu}
\address{Department of Computer Science, University of Nevada, Las Vegas, 4505 S. Maryland Parkway, Las Vegas, NV, 89154-4030, USA}
\email{Farivar@unlv.nevada.edu}

\begin{document}

\maketitle
\begin{abstract}
    In this paper, we study zero divisors in the Hurwitz series rings and the Hurwitz polynomial rings over general non-commutative rings. We first construct Armendariz rings that are not Armendariz of the Hurwitz series type and find various properties of (the Hurwitz series) Armendariz rings. We show that for a semiprime Armendariz of the Hurwitz series type (so reduced) ring $R$ with $a.c.c.$ on annihilator ideals, $HR$ (the Hurwitz series ring with coefficients over $R$) has finitely many minimal prime ideals, say $B_1, \ldots, B_m$ such that $B_1 \cdot \ldots \cdot B_m = 0$ and $B_i = HA_i$ for some minimal prime ideal $A_i$ of $R$ for all $i$, where $A_1, \ldots, A_m$ are all minimal prime ideals of $R$. Additionally, we construct various types of (the Hurwitz series) Armendariz rings and demonstrate that the polynomial ring extension preserves the Armendarizness of the Hurwitz series as the Armendarizness.
\end{abstract}

\section{Introduction}\label{s1}

In an earlier paper \cite{keigher1975}, Keigher studied the properties of formal power series rings and their categorical properties. In the sequel paper \cite{keigher1997}, Keigher introduced an extension of the formal power series ring called the ring of the Hurwitz series and examined its structure and applications, especially in the study of differential algebra. The Hurwitz series rings are similar to formal power series rings, except that binomial coefficients are introduced at each term in the product. While there are many studies of these rings over a commutative ring, very little is known about them over a non-commutative ring. In the present paper, we study the Hurwitz series over a non-commutative ring with identity and examine its structure and properties.

The definition of the Hurwitz series originally allowed the ring to be non-commutative, but most authors restrict them to be commutative; therefore, all of the basic definitions are still true under the restriction that the ring is non-commutative.

The Hurwitz series over any ring $ R $ (not necessarily commutative) is defined as a function from $ N  \to  R $ and denoted by $ HR $. The elements of the ring $ HR $ of the Hurwitz series over $ R $ are sequences of the form $ a = (a_n) = (a_0, a_1, a_2, \ldots) $ where $ a_n \in R $ for each $ n \in \mathbb{N} $. An element in $ HR $ can be thought of as a function from $ \mathbb{N} $ to $ R $. Two elements $ (a_n) $ and $ (b_n) $ in $ HR $ are equal if they are equal as functions from $ \mathbb{N} $ to $ R $, i.e., if $ a_n = b_n $ for all $ n \in \mathbb{N} $. The element $ a_m \in R $ will be called the $ m $-th term of $ (a_n) $. Addition in $ HR $ is defined termwise, so that $ (a_n) + (b_n) = (c_n) $, where $ c_n = a_n + b_n $ for all $ n \in \mathbb{N} $.

If you know the formal power series $ \sum_{n=0}^{\infty} a_n\,t^n \in R[[t]] $ and its coefficients $ (a_n) $, then multiplying in $ HR $ is like multiplying a formal power series, but binomial coefficients are added to each term in the product in the following way: $ C_i^n = \dfrac{n!}{i!\,(n-1)!} $. The Hurwitz product of $ (a_n) $ and $ (b_n) $ is given by $ (a_n) \cdot (b_n) = (c_n) $, where
\[
c_n = \sum_{k=0}^n C_i^n a_k\,b_{n_-k}
\]
Hence,
\[
(a_0, a_1, a_2, a_3, \ldots) \cdot (b_0, b_1, b_2, b_3, \ldots) = 
\]
\[
(a_0b_0, a_0b_1 + a_1b_0, a_0b_2 + 2a_1b_1 + a_2b_0, a_0b_3 + 3a_1b_2 + 3a_2b_1 + a_3b_0, \ldots).
\]

With these two operations, the Hurwitz series ring $ HR $ is a non-commutative ring with identity containing $ R $ The zero in $ HR $ is $ 0 = (0, 0, 0, \ldots) $, the sequence with all terms 0, and the identity is $ 1 = (1, 0, 0, \ldots) $, the sequence with the 0th term 1 and nth term 0 for all $ n \geq 1 $.

A ring $ R $ is called Armendariz \cite{rege1997}, if whenever polynomials $ f(x) = a_0 + a_1x + \ldots + a_nx^n $, $ g(x) = b_0 + b_1x + \ldots + b_mx^m \in R[x] $ satisfy $ f(x)g(x) = 0 $, then $ a_i b_j = 0 $ for each $ i, j $. Throughout this note, all rings are associated with identity. The symbols $ hR $ and $ HR $ stand for the Hurwitz polynomial ring and the Hurwitz series ring, respectively.

A few results on the properties of the Hurwitz series ring are introduced by V. Nourozi et al. \cite{vahid} in Nilpotent Property, and V. Nourozi et al. introduced the McCoy property in \cite{vahid1}. Also, M. Ahmadi et al. introduced the Nilradical property in \cite{vahid2}, and the Armendariz property was introduced by V. Nourozi et al. in \cite{vahid3,vahid4}.

In this paper, we apply the concept of the Armendariz ring to the Hurwitz series rings over general non-commutative rings. $A$ ring $R$ is called Armendariz ring of the Hurwitz series type, if for every series $f= (a_0, a_1, a_2,...)$ and $g = (b_0, b_1, b_2,...)$ in
$HR$, $fg = 0$ if and only if $a_ib_i = 0$ for all $i, j$. Armendariz rings of the Hurwitz
series types are Armendariz by definition; in Section 2, we will show that this direction is nontrivial. A ring is called reduced if it has no nonzero nilpotent
elements. Reduced rings with $char(R) = 0$ are Armendariz rings of the Hurwitz
series type by Lemma 2.3 below, but there may be many examples of nonreduced Armendariz rings of the Hurwitz series types, as we see in Section 3 below. It is obvious that subrings of (the Hurwitz series) Armendariz rings are also (the Hurwitz series)
Armendariz, we will use these facts freely without mention.

\section{Armendariz ring of the Hurwitz series type}\label{s2}

Given a ring $R$, we denote the right (resp. left) annihilator over $R$ by $r_R(-)$ (resp. $l_R(-)$), i.e., $r_R(S) = \{a \in R; Sa = 0\}$ and $l_R(S) = \{b \in R; bS = 0\}$ for a subset $S$ of $R$. If $S$ is a singleton, say $S = \{a\}$, we use $r_R(a)$ (resp. $l_R(a)$) in place of $r_R(\{a\})$ (resp. $l_R(\{a\})$).

Recall that Armendariz rings of the Hurwitz series type are Armendariz; however, the converse need not hold, as shown in the following.
\begin{example}
    Let $F$ be a field. Huh et al. [\cite{huh2002}, Example 14] say that $A = F[a, b, c]$ is the free algebra of polynomials with zero constant terms in noncommuting indeterminates $a, b$ over $F$. They also say that $I$ is an ideal of $F + A$ that is made up of $cc, ac, crc$ with $r \in A$. Next, define $R = F + A/I$. Identify $f$ with $f + I$ for $f \in F + A$ for simplicity. Consider the equality
$
(a - abx)(c + bcx + b^2cx^2 + \ldots + b^ncx^n + \ldots) = a(1 - bx)(1 + bx + b^2x^2 + \ldots + bnx^n + \ldots)c = ac = 0
$
with $1 - bx$ and $1 +bx + b^2x^2+... +b^nx^n + ...$ in $HR$. Since $abc \neq 0$, $R$ is not an Armendariz ring of the Hurwitz series type. But we can prove that $R$ is an Armendariz ring.
\end{example}

A ring is called $Abelian$ if each idempotent is central. Armendariz rings are Abelian by the proof of Anderson and Camillo [~\cite{anderson1998}, Theorem 6]. For a semiprime right Goldie ring $R$, we know that $R$ is Armendariz if and only if it is Armendariz of Huwitz series type if and only if Huh et al. reduce it. [\cite{huh2002}, Corollary 13]: and we also have the same equivalences for von Neumann regular rings by Goodearl [\cite{goodearl1979}, Theorem 3.2] since Armendariz rings are abelian. So one may conjecture that $R$ is Armendariz if and only if $R$ is Armendariz of the Hurwitz series type when $R$ is $\pi$-regular (or strongly $\pi$-regular). However, the following erases the possibility
\begin{example}
    Let $ F $ be a finite field, and $ R $ be the ring in Example 2.1 over $ F $. Then $ R $ is not Armendariz of the Hurwitz series type but Armendariz by the arguments in Example 2.1. It is easy to check that $ R $ is a locally finite ring (i.e., each finite subset generates a finite multiplicative semigroup). So for any $ r \in R $, there is a positive integer $ k $ such that $ r^k $ is an idempotent by the proof of Huh et al. [~\cite{huh2002}, Proposition 16], implying that $ R $ is strongly $\pi$-regular (hence $\pi$-regular).
\end{example}

Due to Bell \cite{bell1970}, a right (or left) ideal $ I $ of a ring $ R $ is said to have the insertion-of-factors-principle (simply, IFP) if $ ab \in I $ implies $ aRb \subseteq I $ for $ a, b \in R $. So we shall call a ring $ R $ an IFP ring if the zero ideal of $ R $ has the IFP. Shin \cite{shin1973} used the term SI for the IFP. Computations show that reduced rings are IFP, and IFP rings are Abelian. Subrings of IFP rings are also IFP, obviously.

Note that a ring $ R $ is IFP if and only if any right annihilator is an ideal if and only if any left annihilator is an ideal if and only if $ ab = 0 $ implies $ aRb = 0 $ for $ a, b \in R $ Shin [~\cite{shin1973}, Lemma 1.2)].

Given a ring $ R $, we use $ N_*(R) $, $ N^*(R) $, $ N_0(R) $, and $ N(R) $ to represent the prime radical (i.e., lower nilradical), the upper nilradical (i.e., the sum of all nil ideals), the Wedderburn radical (i.e., the sum of all nilpotent ideals), and the set of all nilpotent elements of $ R $, respectively. Note $ N_0(R) \subseteq N_*(R) \subseteq N^*(R) \subseteq N(R) $.

\begin{lemma}
    [\cite{vahid}, corollary 2.6] Every reduced ring with $ \text{char}(R) = 0 $ is an Armendariz ring of the Hurwitz series type.
\end{lemma}
\begin{proposition}
    Armendariz rings of the Hurwitz series type are IFP.
\end{proposition}
 
\begin{proof}
    Let $R$ be an Armendariz ring of Hurwitz series type and suppose that $ab = 0$ for $a,b \in R$. For any $r \in R$ we have $a(1-rx)(1+1!r2x+2!+r^2x^2 +3!r^3x^3 +\cdots+ n!r^nx^n +\cdots)b = ab = 0$, where $1-rx$ and $1+1!rx+2!r^2x^2 +3!r^3x^3 +\cdots+ n!r^nx^n + \cdots $are in $HR$. Thus, the Armendariz ring of the Hurwitz series type $R$ implies $aRb = 0$, proving that $R$ is $IFP$.\end{proof}

\begin{proposition}
    If $R$ is an Armendariz ring of the Hurwitz series type then $HR$ is $IFP$.
    \end{proposition}

\begin{proof}
    If $R$ is an Armendariz ring of the Hurwitz series type, then it is $IFP$ by the result of Proposition 2.4. Let $f = \sum_{i=0}^{\infty} a_i\,X^i , g = \sum_{j=0}^{\infty} b_j\,X^j \in HR$ such that $fg = 0$. Then $a_ib_j = 0$ for all $i,j$ since $R$ is armendariz of theHurwitz series type, but $R$ is $IFP$ and so $a_iRb_j = 0$ for all $i,j$. Consequently, we have $fHRg = 0$, proving that $HR$ is $IFP$.
    \end{proof}
\begin{proposition} \hspace{3cm}
    \begin{itemize}
        \item[(1)] [\cite{shin1973}, Theorem 1.5] If $ R $ is an IFP ring, then $ N_*(R) = N(R) $.
        \item[(2)] If $ R $ is an Armendariz ring, then $ N_0(R) = N_{ast}(R) = N^*(R) $.
        \item[(3)] If $ R $ is an Armendariz ring of the Hurwitz series type, then $ N_0(R) = N_*(R) = N^*(R) = N(R) $.
    \end{itemize}
\end{proposition}

\begin{proof} \hspace{3cm}
    \begin{itemize}
        \item[(1)] We provide a proof simpler than that of Shin (\cite{shin1973}, $Theorem 1.5)$. Let $ R $ be an IFP ring, and take $ a \in N(R) $ with $ a^n = 0 $ for some positive integer $ n $. Then, since $ R $ is IFP, we have $ aR_1aR_2a\dots aR_{n1}a = 0 $ with $ R_i = R $ for $ i = 1, 2, \ldots, n1 $, hence $ a $ is strongly nilpotent.
        \item[(2)] Suppose that $ ab = 0 $ and $ ac^n\,b = 0 $ for $ a, b, c \in R $ and some integer $ n \geq 1 $. Then we have $ acb = 0 $ from 
    $a(1 - cx)(1 + 1!\,cx + \ldots + (n - 1)!\,c^{n-1}\,x^{n-1})b = 0$ Because $ R $ is Armendariz, entailing $ aSb = 0 $ for $ S \subseteq N(R) $. Let $ d \in N^*(R) $ with $ d^m = 0 $, and consider the powers of $ RdR $. If $ m = 2 $, then $ (RdR)^{22-1} = RdRdRdR = 0 $ since $ RdR \subseteq N^*(R) $. If $ m = 3 $, then $ Rd^2RdRdR = 0 $, and $ (RdR)^{23-1} = RdRdRdRdRdR = 0 $ since $ RdR \subseteq N^*(R) $. Inductively, we get $ (RdR)^{2m-1} = 0 $, causing $ d \in N_0(R) $.        \item[(3)] This result comes from $(i), (ii)$, and Proposition 2.4.
    \end{itemize}    
\end{proof}

An ideal $ I $ of a ring $ R $ is called completely prime if $ R/I $ is a domain, while it is called $completely$ $semiprime$ if $ R/I $ is reduced.
\begin{remark}
    \hspace{3cm}
    \begin{itemize}
        \item[(1)] Armendariz rings of the Hurwitz series type need not be reduced, as shown by examples in Section 3.
        \item[(2)] IFP rings need not be Armendariz (and thus not be Armendariz of the Hurwitz series type) by the following examples:
        \begin{itemize}
             \item[\*] Rege and Chhawchharia (\cite{rege1997}, Example 3.2), Example 3.2, provides a commutative ring (and thus IFP) that is not Armendariz.
        \end{itemize}
        \item[(3)] The converses of Proposition 2.6(2,3) need not hold, as can be seen by the 2-by-2 upper triangular matrix ring over a reduced ring.
        \item[(4)] Shin showed that every minimal prime ideal of a ring $ R $ is completely prime if and only if $ N_*(R) = N(R)$ (\cite{shin1973}, Proposition 1.11). Thus, Proposition 2.4 and Proposition 2.6 $(i)$ imply that if $ P $ is a minimal prime ideal of an Armendariz ring of the Hurwitz series type $ R $, then $ R/P $ is a domain (so Armendariz of the Hurwitz series type). However, every prime factor ring of Armendariz rings (Hurwitz series) need not be Armendariz (the Hurwitz series), as shown by Goodearl and Warfield (\cite{goodearl1989}, Exercise 2A). Let $ R $ be the ring of quaternions with integer coefficients; then $ R $ is a domain and thus Armendariz (the Hurwitz series). For any odd prime integer $ q $, the factor ring $R/qR$ is isomorphic to the $2 \times 2$ matrix ring over the field of integers modulo $q$, as computed in Goodearl and Warfield (\cite{goodearl1989}, Exercise 2A). However, $R/qR$ is not Abelian and thus not Armendariz.
    \end{itemize}
\end{remark}

According to Huh et al., Armendariz rings do not necessarily need to be IFP. (\cite{huh2002}, Example 14). However, one may conjecture that IFP Armendariz rings are Armendariz of the Hurwitz series type. Nevertheless, there are commutative Armendariz rings that are not Armendariz of the Hurwitz series type, as demonstrated in the following.

\begin{example}
    Let $F$ be a field of characteristic $\neq 2$, and $S = F[A]$ be the polynomial ring generated by the commuting indeterminates $A = \{a_i\}_{i=0}^\infty$ over $F$. Let $I$ be the ideal of $S$ generated by $\{a_ia_ja_k \mid a_i,a_j,$ and $a_k$ are in $A\}$, and put $T = S/I$. Let $d_n = \sum_{i+j=n}a_i\,a_j$ and $J$ be the ideal of $I$ generated by $\{d_n \mid n = 0,1,\ldots\}$. The remainder of the construction is the same as in Example 2.1. Then $R$ is a commutative ring, and it is Armendariz but not Armendariz of Hurwitz series type by the computations of Example 2.1.
\end{example}

\begin{corollary}
    Let $R$ be an Armendariz ring of the Hurwitz series type. Then $R$ is semiprime if and only if $R$ is reduced.
\end{corollary}

\begin{proof}
    It suffices to prove the necessity. Since $R$ is IFP by Proposition 2.4, the prime radical of $R$ contains all nilpotent elements in $R$ by Proposition 2.3(1). Thus, $R$ is reduced since it is semiprime by the condition.
\end{proof}

Hirano \cite{hirano2002} observed relations between annihilators in a ring $R$ and annihilators in $R[x]$. In this paper, we also extend Hirano's results for polynomial rings to the
situation of the Hurwitz series rings. Given a ring $R$, we define $r Ann_R(2^R)$ = $\{r_R(U) \mid U \subseteq R\}$, $r Ann_{HR}(2^{HR})$ = $\{r_{HR}(V) \mid V \subseteq {HR}\}$,$\quad lAnn_R(2^R) = \{l_R(U) \mid U \subseteq {R}\}.$ Given a the Hurwitz series $f \in {HR}$, let $C_f$ denote the set of all coefficients of $f$, and for a subset $V$ of $HR$, let $C_V$ denote the set $U_{f \in v}C_f$.

\begin{remark}
    Note that $r_{HR}(V)R = r_R(V) = r_R(C_v)$ for $V \subseteq HR$, so we can define a map $\Psi: rAnn_{HR}(2^{HR}) \to rAnn_{R}(2^R)$ by $\Psi(I) = I\cap R$ for each $J \in rAnn_{HR}(2HR)$. It is easily shown that $\Psi$ is surjective.
\end{remark}

Let $U \subseteq R$. It is easy to check $H(r_R(U)) \subseteq r_{HR}(U)$. Conversely, letting $Uf = 0$ with $f \in HR$, then $Ua_i = 0$ (i.e., $a_i \in r_R(U)$) for all $i$, and thus $f \in H(r_R(U))$, hence we have $H(r_R(U)) = r_{HR}(U)$. So we can define a map $\Phi : rAnn_{R}(2^R) \to rAnn_{HR}(2^{HR})$ by $\Phi(J) = HJ$ for each $J \in rAnn_{R}(2^R)$. It is trivial that $\Phi$ is injective.

When given rings are Armendariz of the Hurwitz series type, $\Phi$ is surjective (if and only if $\Psi$ is injective), as can be seen in the following result. In this situation, $\Phi$ and $\Psi$ are inverses of each other.

\begin{proposition}
    Given a ring $R$ the following conditions are equivalent:
\end{proposition}

    \begin{itemize}
        \item[(1)] R is Armendariz of the Hurwitz series type;
        \item[(2)] If $f_1 \ldots f_n = 0$ for $f_1, \ldots, f_n \in {HR}$, then $a_1 \ldots a_n = 0$ where $a_i$ is any coefficient of $f_i$ and $n$ is any integer $\geq 2$.
        \item[(3)] $\Phi : {\it rAnn}_{R}(2^{R}) \to {\it rAnn}_{HR}(2^{HR})$; $J \mapsto HJ$ is surjective (so bijective).
        \item[(4)] $\Phi': {\it lAnn}_{R}(2^R) \to {\it lAnn}_{HR}(2^{HR})$; $K \mapsto HK$ is surjective (so bijective).
        \item[(5)] $\Psi: {\it rAnn}_{HR}(2^{HR}) \to {\it rAnn}_{R}(2^{R})$; $I \mapsto I \cap R$ is injective (so bijective).
        \item[(6)] $\Psi': {\it lAnn}_{HR}(2^{HR}) \to {\it lAnn}_{R}(2^{R})$; $H \mapsto H\cap R$ is injective (so bijective).
    \end{itemize}

\begin{proof}
    The proof of (2) $\Rightarrow$ (1) is trivial. (1) $\Rightarrow$ (2): Let $R$ be Armendariz of the Hurwitz series type and assume that $f_1 \ldots f_n = 0$ for $f_1, \ldots, f_n \in HR$ with $a_i$ being any coefficient of $f_i$ and $n$ being any integer $\geq 2$. From $f_1(f_2 \ldots f_n) = 0$, we get $a_1\alpha = 0$ for any coefficient $\alpha \in f_2 \ldots f_n$ inducing $a_1f_2 \ldots f_n,$ including $a_1a_2f_3 \ldots f_3 = 0.$ Inductive computation yields the result.
\end{proof}

Next, we use a similar method to Hirano (\cite{hirano2002}, Proposition 3.1) for the proof of (1) $\Rightarrow$ (3).

(1) $\Rightarrow$ (3) Let $f = \sum_{i=0}^{\infty} a_iX^i \in r^{HR}(V) \in rAnn_{HR}(2^{HR})$ and suppose that $R$ is Armendariz of the Hurwitz series type. Then we have $b_ja_i = 0$ for all $i$ and $j$ where $\sum_{j=0}^{\infty} b_jX^j$ is any Hurwitz series in $V$. So $a_i \in r_R(C_V)$ and $f \in H(r_R(C_V))$, hence $r_{HR}(V) = H(r_R(C_V)) = \Phi(r_R(C_V))$, implying that $\Phi$ is surjective.

(3) $\Rightarrow$ (1) Suppose that $\Phi$ is surjective and that $fg = 0$ for $f = \sum_{i=0}^{\infty} a_iX^i $ and $g = \sum_{j=0}^{\infty} b_jX^j \in {HR}$. Then for $r_{HR}(f)$, there exists $r_R(F) \in rAnn_R(2^R)$ such that $H(r_R(F)) = \Phi(r_R(F)) = r_{HR}(f)$. So $g \in r_{HR}(f)$ implies that each $b_j$ is contained in $r_R(F)$, hence $fb_j = 0$ and $a_ib_j = 0$ for all $i$ and $j$.

(1) $\Leftrightarrow$ (4) is shown by the left version of the preceding proofs. (1) $\Leftrightarrow$ (5) and (1) $\Leftrightarrow$ (6) are immediate consequences of (1) $\Leftrightarrow$ (3) and (1) $\Leftrightarrow$ (4), respectively.

Due to Kaplansky \cite{kaplansky1968}, a ring $R$ is called Baer if the right annihilator of every nonempty subset of $R$ is generated by an idempotent. The concept of the Baer ring is left-right symmetric by Kaplansky (\cite{kaplansky1968}, Theorem 3). A ring $R$ is called {\it right (left) p.p.} if each principal right (left) ideal of $R$ is projective. A ring $R$ is called {\it p.p.} if it is a right and left {\it p.p.} ring. Baer rings are clearly {\it p.p.} rings.

\begin{lemma}
    Suppose that a ring $R$ is abelian. Then we have the following:
        \begin{itemize}
        \item[(1)] Every idempotent of $hR$ is in $R$ and $hR$ is abelian.
        \item[(2)] Every idempotent of $HR$ is in $R$ and $HR$ is abelian.
        \end{itemize}
\end{lemma}

\begin{proof}
    $hR$ is a sub-ring of $HR$, and so it is enough to prove (2). For $f \in HR$, assume that $f = f^2$, where $f = e_0 + e_1x + \ldots + e_nx^n + \ldots$. Then we have the system of equations
\end{proof}

\begin{equation}
    e_0^2 = e_0,
\end{equation}
\begin{equation}
    e_0e_1 + e_1e_0 = e_1,
\end{equation}
\begin{equation}
    e_0e_2 + 2e_1e_1 + e_2e_0 = e_2,
\end{equation}
$$ \vdots $$

Observing that Eq. (2.1) yields that $e_0$ is an idempotent, so it is central. If we multiply Eq. (2.2) on the left side by $e_0$, then $e_0e_1 + e_0e_1e_0 = e_0e_1$. But $e_0e_1e_0 = e_0e_1$ because $e_0$ is central. So $e_0e_1 = 0$ and so $e_1 = 0$, hence Eq. (2.3) becomes $e_0e_2 + e_2e_0 = 0 = e_2$. If we multiply Eq. (2.3) on the left side by $e_0$, then $e_0e_2 + e_0e_2e_0 = e_0e_2$. But $e_0e_2e_0 = e_0e_2$. Hence $e_0e_2 = 0$ and so $e_2 = 0$. Continuing this way, we get that $e_i = 0$ for all $i \geq 1$. Therefore, $f = e_0 \in R$, and also $HR$ is abelian.
\begin{corollary}
    Let $R$ be an Armendariz ring of Hurwitz series type. If $e = (e_0, e_1, \ldots) \in HR$ is an idempotent, then $e = e_0 \in R$.
\end{corollary}

\begin{theorem}
    Let $R$ be an Armendariz ring of Hurwitz series type. Then $R$ is a Baer ring if and only if $HR$ is a Baer ring.
\end{theorem}

\begin{proof}
    Suppose that $R$ is a Baer ring and $A$ is a subset of $HR$. Let $B$ be the set of all coefficients of all elements of $A$. Then we have $r_R(B) = e_0R$ for some idempotent $e_0 \in R$. We prove that $r_{HR}(A) = eHR$ where $e = (e_0, 0, 0, \ldots)$ is an idempotent in $HR$. We know that $eHR \subseteq r_{HR}(A)$. Next, if $g = (b_0, b_1, b_2, \ldots) \in A$, then $fg = 0$. Since $R$ is Armendariz of Hurwitz series type, $a_ib_j = 0$ for each $i, j$. Thus $b_j \in r_R(B) = e_0R$, so $b_j = e_0b_j$, which implies $g = eg$, and hence $HR$ is a Baer ring.
\end{proof}

Conversely, if $HR$ is Baer and $A \subseteq R$, then we have $r_{HR}(A) = eHR$, for some idempotent $e = (e_0, 0, 0, \ldots) \in HR$, by Corollary 2.13. Thus, we get $r_R(A) = eHR \cap R = e_0R$.

\begin{theorem}
    Let $R$ be an Armendariz ring of Hurwitz series type. Then $R$ is a p.p.-ring if and only if $hR$ is a p.p.-ring.
\end{theorem}

\begin{proof}
    Suppose that $R$ is a p.p.-ring and $f \in ThR$. Let $\{a_i \mid 0 \leq i \leq n\}$ be the set of all coefficients of $f$. Then we have $r_R(a_i) = e_iR$ for some idempotent $e_i \in R$ with $0 \leq i \leq n$. Put $e = e_1e_2 \dots e_n$. Since $R$ is commutative, $e$ is an idempotent of $R$. We prove that $r_{hR}(f) = ehR$. We know that $ehR \subseteq r_{hR}(f)$. Assume that $g \in r_{hR}(f)$, so $fg = 0$ and $a_ib_j = 0$ for every $0 \leq i \leq n$ and $0 \leq j \leq m$ and for each coefficient $b_i$ of $g$. Thus $b_j = e_ib_j$ for each $0 \leq i \leq n$ and $0 \leq j \leq m$, and so $b_j = eb_j$ implies $g = eg$. Consequently, $g = eg$ and hence $g \in ehR$. Therefore, $hR$ is p.p.
\end{proof}

Conversely, assume that $hR$ is a p.p.-ring and that $a \in R$. So $r_{hR}(ahR) = ehR$ for some idempotent $e \in hR$. By Corollary 2.13, $e \in R$ is an idempotent, so $r_R(aR) = ehR \cap R = eR$, then $r_R(a) = eR$ and hence $R$ is p.p.

A ring $R$ is called left (resp. \emph {right}) {\it weakly $\pi$-regular} if for every $a \in R$ there exists a positive integer $n$, depending on $a$, such that $a^n \in (Ra^n)^2$ (resp. $a^n \in (a^nR)^2$). Any $\pi$-regular ring is both left and right weakly $\pi$-regular.

\textbf{Proposition 2.16. }{\it If $R$ is an Armendariz ring of the Hurwitz series type such that $R/N_0(R)$ is right (or left) weakly $\pi$-regular, then every prime ideal of $R$ is both maximal and completely prime.}

\textit{Proof.} Since $R$ is Armendariz of the Hurwitz series type, $N_0(R) = N_*(R) = N^*(R) = N(R)$ by Proposition 2.6(3), hence $R/N_*(R)$ is right (or left) weakly $\pi$-regular.
Every prime ideal of $R$ is maximal by Birkenmeier et al.~\cite{birkenmeier1994} (Lemma 5) and thus completely prime by Shin~\cite{shin1973} (Proposition 1.11).

\section{Example of Armendariz ring of the Hurwitz series type}\label{s3}

In this section, we study various examples of Armendariz rings of the Hurwitz series type, extending the class to the realm of nonreduced non-commutative rings. We denote the degree of a polynomial $f$ by $\deg f$.

\begin{proposition}
    A ring $R$ is Armendariz of the Hurwitz series type if and only if $hR$ is Armendariz of Hurwitz series type.
\end{proposition}

\begin{proof}
    Let $R$ be an Armendariz ring of the Hurwitz series type, and suppose that $f(T) = \sum_{i=0}^{\infty} f_iT^i$ and $g(T) = \sum_{j=0}^{\infty} g_jT^j$ are in $H(hR)$ with $f(T)g(T) = 0$, where $f_i, g_j \in hR$ and $T$ is an indeterminate over $hR$. We apply a similar proof to Anderson and Camillo (\cite{anderson1998}, Theorem 2). Set $k_n = \deg f_0 + \ldots + \deg f_n + \deg g_0 + \ldots + \deg g_n + 1$, and the degree of the zero is taken to be zero. Now define $f(X) = f_0 + f_1x^{k_1} + f_2x^{2k_2} + \ldots + f_nx^{n k_n} + \ldots$ and $g(X) = g_0 + g_1x^{k_1} + g_2x^{2k_2} + \ldots + g_nx^{n k_n} + \ldots$. Then $f(X), g(X)$ are contained in $HR$, and the set of coefficients of $f_i$'s (resp. $g_j$'s) is equal to the set of coefficients of $f(X)$ (resp. $g(X)$). From $f(T)g(T) = 0$, it gives $f(X)g(X) = 0$ by the definition of $k_n$'s. Since $R$ is Armendariz of the Hurwitz series type by hypothesis, each coefficient of $f_i$ annihilates each coefficient of $g_j$. Thus, $f_ig_j = 0$ for all $i$ and $j$.
\end{proof}

\begin{proposition}
    If $R$ is a principal ideal domain, then every homomorphic image of $R$ is Armendariz of the Hurwitz series type.
\end{proposition}

\begin{proof}
    Let $I$ be a nonzero proper ideal of $R$ with $I = (p_1^{m_1} \ldots p_k^{m_k})R$, where $p_i$'s are distinct primes and $m_i \geq 1$ for all $i, k \geq 1$. Since $I = p_1^{m_1} \cap \ldots \cap p_k^{m_k} \cap R$, $R/I$ is isomorphic to $\dfrac{R}{p_1^{m_1}R} \oplus \ldots \oplus \dfrac{R}{p_k^{m_k}R}$ by the Chinese remainder theorem. Applying a similar proof to that of Rege and Chhawchharia~\cite{rege1997} (Proposition 2.1) to this situation, we can obtain that each $R/p_i^{m_i} R$ is Armendariz of the Hurwitz series type. It is trivial that the direct product of Armendariz rings of the Hurwitz series type is also Armendariz of the Hurwitz series type; hence, so is $R/I$.
\end{proof}

By Proposition 3.2, every factor ring of a polynomial Hurwitz ring over a field is Armendariz of the Hurwitz series type; however, this result need not hold for a polynomial Hurwitz ring over any commutative domain, as in the following remark.
\begin{remark}
    According to Anderson and Camillo~\cite{anderson1998} (Example 10), let $k$ be a field and $R = k[x, y] = k[x][y]$ be the polynomial ring with commuting indeterminates $x, y$ over $k$, Next, consider $S = \dfrac{k[x, y]}{(x^2, y^2)}$ with $(x^2, y^2)$ being the ideal of $R$ generated by $x^2$ and $y^2$. Then $S$ is not Armendariz by the equality $(x + yT)(x - yT) = 0$, where $T$ is an indeterminate over $S$. Applying the method in the proof of Anderson and Camillo~\cite{anderson1998} (Theorem 5) to the case of the Hurwitz series ring, we obtain the following proposition.
\end{remark}

\begin{proposition}
    Let $R$ be a ring and $n$ be a positive integer $\geq 2$. Then the following conditions are equivalent:
    \begin{itemize}
        \item[(1)] $R$ is reduced;
        \item[(2)] ${hR}/{(x^n)}$ is Armendariz;
        \item[(3)] ${hR}/{(x^n)}$ is Armendariz of the Hurwitz series type, where $(x^n)$ is the ideal of $hR$ generated by $(x^n)$.
    \end{itemize}
\end{proposition}

\begin{proposition} \hspace{3cm}
        \begin{itemize}
        \item[(1)] Let $R$ be a domain and $M$ be an $(R,R)$-bimodule. Then $T(R,M)$ is Armendariz of Hurwitz series type if and only if so is $M$.
        \item[(2)] A ring $R$ is reduced if and only if $T(R, R)$ is Armendariz of the Hurwitz series type if and only if $T(R, R)$ is Armendariz.
    \end{itemize}
\end{proposition}

\begin{proof}
    \begin{itemize}
        \item[(1)] By a similar proof to Anderson and Camillo~\cite{anderson1998} (Theorem 12).
        \item[(2)] By Lee and Wong~\cite{lee2003} (Theorem 2.3) and a similar proof to Rege and Chhawchharia~\cite{rege1997} (Proposition 2.5).
    \end{itemize}
\end{proof}

Next, we extend Lee and Wong~\cite{lee2003} (Theorem 2.4) onto the Hurwitz series rings. Let $A$ be a subset of the set $\{1, 2, \ldots, n\}$, and let $\psi: A \to \{1, 2, \ldots, n\}$ be an arbitrary function. By the same method in the proof of Lee and Wong~\cite{lee2003} (Theorem 2.4), we obtain the following proposition.

\begin{proposition}
    Let $R = R_0 \oplus R_1 \oplus R_2 \oplus R_3 \oplus \ldots$, where each $R$ is an additive subgroup of $R$ such that $R_iR_j \subseteq R_{i+j}$ for all $i, j$, and $R_k = 0$ for $k \geq 3$. Suppose that $R_1$ is a free $R_0$-module with a basis $v_1, v_2, \ldots, v_{n}$ such that $[v_i, R_0] = 0$ for each $i$ and $v_i v_j \neq 0$ if and only if $j = \psi(i)$, and that the set $\{v_i v_{\psi(i)} \mid i \in A\}$ forms a free basis of an $R_0$-module $R_2$. Then the following conditions are equivalent:
    \begin{itemize}
        \item[(1)] $R_0$ is reduced;
        \item[(2)] $R$ is an Armendariz ring;
        \item[(3)] $R$ is an Armendariz ring of the Hurwitz series type.
    \end{itemize}
    
\end{proposition}

\begin{corollary}
    Let $S$ be a reduced ring and $I$ be an ideal of $S$.
    \begin{itemize}
        \item[(1)] Suppose that $S/I$ is a reduced ring. Then
        \[
          R = \left\{ \begin{pmatrix} a & f & g \\ 0 & a & h \\ 0 & 0 & a \end{pmatrix} \vert a \in S \text{ and } f, g, h \in \dfrac{S}{I} \right\}
        \]
          is an Armendariz ring of the Hurwitz series type with usual matrix addition and multiplication.
        \item[(2)] $\left\{ \begin{pmatrix} a & b & c \\ 0 & a & d \\ 0 & 0 & a \end{pmatrix} \bigg\vert a, b, c, d \in S \right\} $ is an Armendariz ring of the Hurwitz series type.
        \item[(3)] $T(S, \dfrac{S}{I})$ {\it  is an Armendariz ring of the Hurwitz series type when } $\dfrac{S}{I}$ \text{ is a reduced ring.}
    \end{itemize}
\end{corollary}

\begin{proof}
        \begin{itemize}
        \item[(1)] Let $n=2$, $A=\{1\}$, and $\psi(1)=2$. Next, set
\begin{equation*}
               R_0 = \left\{\begin{pmatrix} a & 0 & 0 \\ 0 & a & 0 \\ 0 & 0 & a \end{pmatrix} \in R \right\} 
\end{equation*}
\begin{equation*}
               R_1 = \left\{\begin{pmatrix} 0 & f & 0 \\ 0 & 0 & h \\ 0 & 0 & 0 \end{pmatrix}\in R \right\} 
\end{equation*}
\begin{equation*}
               R_2 = \left\{\begin{pmatrix} 0 & 0 & g \\ 0 & 0 & 0 \\ 0 & 0 & 0 \end{pmatrix}\in R \right\} 
\end{equation*}
        \[
        \upsilon_1 = \begin{pmatrix} 0 & 1+I & 0 \\ 0 & 0 & 0 \\ 0 & 0 & 0 \end{pmatrix}, \quad
        \text{and} \quad
        \upsilon_2 = \begin{pmatrix} 0 & 0 & 0 \\ 0 & 0 & 1+I \\ 0 & 0 & 0 \end{pmatrix}.
        \]
        Then Proposition 3.6 applies.
        \item[(2)] Set $I$ = 0 in (1).
        \item[(3)] $T(S, \dfrac{S}{I})$ is isomorphic to a subring of the ring $R$ in Proposition 3.6.
    \end{itemize}
\end{proof}

\begin{remark}
    \begin{itemize}
        \item[(1)] Based on Corollary 3.7(2), one may suspect that the ring
    \[
    R_n = \left\{ \begin{pmatrix} a & a_{12} & a_{13} & \cdots & a_{1n} \\ 0 & a & a_{23} & \cdots & a_{2n} \\ 0 & 0 & a & \cdots & a_{3n} \\ \vdots & \vdots & \vdots & \ddots & \vdots \\ 0 & 0 & 0 & \cdots & a \end{pmatrix} \bigl\vert a,a_{ij} \in S \right\}
    \]
    It may also be an Armendariz ring of the Hurwitz series type for $n \geq 4$ when $S$ is a reduced ring.\\
        \item[(2)] From Corollary 3.7(3), one may also suspect that if $R$ is Armendariz of the Hurwitz series type, then so is the trivial extension $T(R, R)$.
    \end{itemize}
\end{remark}

The ring $R$ in Example 2.1 is a non-commutative local ring that is not Armendariz of the Hurwitz series type, with $(A)$ being the maximal ideal, where $(A)$ is the ideal generated by $A$. However, we observe a class of local rings that may be Armendariz of Hurwitz series type as follows. An element $a$ in a ring $R$ is called regular if $r_R(a) = 0 = l_R(a)$, i.e., $a$ is not a zero divisor.

\begin{proposition}
    Let $R$ be a ring and $J$ be an ideal of $R$ such that every element in $R/J$ is regular and $J^2 = 0$. Then $R$ is an Armendariz ring of the Hurwitz series type.
\end{proposition}

\begin{proof}
    Let $fg = 0$ for $f = \sum_{i=0}^{\infty} a_iX^i$ and $g = \sum_{j=0}^{\infty} b_jX^j$ in $HR$. We may assume that $a_0$ and $b_0$ are both nonzero. Since $R/J$ is a domain, Armendariz (~\cite{armendariz1974}, Lemma 1) implies that $a_ib_j \in J$ for all $i$ and $j$. In the following computation, we use the condition that $R/J$ is the set of all regular elements without mention. $a_0b_0 = 0$ gives $a_0, b_0 \in J$. $a_0b_1 + a_1b_0 = 0$ gives that $a_1$ is regular if and only if $b_1$ is regular. For, if $a_1$ is regular, then $a_1b_0 \neq 0$ and so $a_0b_1 \neq 0$. This implies that $b_1 \notin J$ because $J^2 = 0$, and so $b_1$ is regular. The converse is proven using the same method. Hence $a_1, b_1 \in J$ since $a_1b_1 \in J$. Summarizing, we have $a_ib_j = 0$ for all $i,j \in \{0,1\}$.
    We proceed by induction on $n$. Consider $a_0b_n + (n-1)a_1b_{n-1} + \dots + (n-1)a_{n-1}b_1 + a_nb_0 = 0$, then $a_i, b_j \in J$ for $i,j \in \{0, 1, \dots, n-1\}$ by the induction hypothesis and so $J^2 = 0$ implies $a_0b_n + a_nb_0 = 0$. So $a_nb_n \in J$ gives $a_n, b_n \in J$ by the above computation. Consequently, $a_ib_j = 0$ for all $i$ and $j$ since $J^2 = 0$.
\end{proof}
 
Let $R$ be a commutative ring and $h$ be a ring endomorphism of $R$. For an $R$-module $M$, the multiplication $(a, m)(b, n) = (ab, h(a)n + bm)$ gives a ring structure to $R \oplus M$, denoted by $R(+)hM$.

\begin{corollary}
    \begin{itemize}
        \item[(1)] (Anderson and Camillo, ~\cite{anderson1998}, Theorem 12) Let $R$ be a commutative domain and $M$ be an $R$-module. If $M$ is torsion-free, then the trivial extension $T(R,M)$ is an Armendariz ring (of Hurwitz series type).
        \item[(2)] (Rege and Chhawchharia, ~\cite{rege1997}, Proposition 2.8) Let $K$ be a field, $h$ be a nonzero ring endomorphism of $K$, and $V$ be a $K$-vector space. Then $K(+)hV$ is an Armendariz ring (of the Hurwitz series type).
        \item[(3)] Local rings are Armendariz of the Hurwitz series type when $J^2 = 0$, where $J$ is the Jacobson radical.
    \end{itemize}
\end{corollary}

\begin{proof}
    \begin{itemize}
        \item[(1)] Let $J = \{(a, m) \in T(R,M) \mid a = 0\}$, then $J^2 = 0$ and $R/J$ is the set of all regular elements in $T(R,M)$.
        \item[(2)] Let $J = \{(k, \nu) \in K(+)_hV \mid k = 0\}$, then $J^2 = 0$ and $R/J$ is the set of all regular elements in $K(+)_hV$.
        \item[(3)] Let $R$ be a local ring. Then $R/J$ is clearly the set of all regular elements in $R$.
    \end{itemize}
\end{proof}

However, the ring $R$ in Proposition 3.9 need not be Armendariz (of the Hurwitz series type) when $J_n = 0$ with $n \geq 3$ by the following example.
\begin{example}
    \begin{itemize}
        \item[(1)] (Anderson and Camillo, \cite{anderson1998}, Example 10) and Example 2.1 are counterexamples to the case of $J^3 = 0$.
        \item[(2)] Denote the Jacobson radicals of the following rings by $J$. Let $F$ be a simple domain and $R = \dfrac{F[x, y]}{(x^m,y^m)}$ with $m \geq 2$, where $F[x, y]$ is the polynomial ring with $x, y$ commuting indeterminates over $F$, and $(x^m, y^m)$ is the ideal of $F[x, y]$ generated by $x^m, y^m$. Then $J^{2m-1} = 0$ with $J = xR + yR$. Let $T$ be an indeterminate over $R$. If $m$ is even, then the equality $(x^\frac{m}{2} + y^\frac{m}{2} T)(x^\frac{m}{2} - y^\frac{m}{2} T) = 0$ implies that $R$ is not Armendariz. Next, let $S = \frac{F[x, y]}{(x^{m+1}, y^m)}$ with $m \geq 2$, where $(x^m, y^m)$ is the ideal of $F[x, y]$ generated by $x^{m+1}, y^m$. Then $J^{2m} = 0$ with $J = xS + yS$. Let $T$ be an indeterminate over $S$. If $m$ is even, then the equality $(x^{\frac{m}{2} + 1} + y^{\frac{m}{2}} T)(x^{\frac{m}{2} + 1} - y^{\frac{m}{2}} T) = 0$ implies that $S$ is not Armendariz. If $m$ is odd, then the equality $(x^{\frac{m+1}{2}} + y^{\frac{m+}{2}} T)(x^{\frac{m+1}{2}} - y^{\frac{m+1}{2}} T) = 0$ implies that $S$ is not Armendariz.
    \end{itemize}
\end{example}

The following may be convenient methods to check whether the given rings are Armendariz of Hurwitz series type if they are available.

\begin{corollary}
    \begin{itemize}
        \item[(1)] Let $R$ be an Abelian ring and $\text{char}(R) = 0$. Then $R$ is Armendariz of the Hurwitz series type if and only if $eR$ and $(1 - e)R$ are Armendariz of the Hurwitz series type for every idempotent $e$ of $R$ if and only if $eR$ and $(1 - e)R$ are Armendariz of Hurwitz series type for some idempotent $e$ of $R$.
        \item[(2)] For a ring $R$, suppose that $R/I$ is an Armendariz ring of the Hurwitz series type for some ideal $I$ of $R$ and $\text{char}(R/I) = 0$. If $I$ is a reduced ring, then $R$ is Armendariz of Hurwitz series type.
    \end{itemize}
\end{corollary}

\begin{proof}
    \begin{itemize}
        \item[(1)] can be proved by a similar proof of Huh et al. (~\cite{huh2002}, Proposition 10).
        \item[(2)] Let $f = \sum_{i=0}^{\infty} a_iX^i$, $g = \sum_{j=0}^{\infty} b_jX^j \in HR$ such that $fg = 0$. Since $R/I$ is Armendariz of Hurwitz series type and $\text{char}(R/I) = 0$, we have $a_ib_j \in I$ for all $i,j$. First, we show that $a_0b_j = 0$ for all $j$. Assume, on the contrary, that $a_0b_n \neq 0$ for some $n$ (let $n$ be the smallest one relating to this property). Using the condition that $I$ is reduced and the proof of Huh et al. (~\cite{huh2002}, Theorem 11), we get a contradiction. Therefore, $a_0b_j = 0$ for all $j$. Consequently, $0 = fg = \sum_{i=1}^{\infty} a_iX^{i} \sum_{j=0}^{\infty} b_jX^j = \sum_{i=1}^{\infty} a_iX^{i-1}\sum_{j=0}^{\infty} b_jX^j = 0$. Next, we can obtain $a_1b_j = 0$ since $\text{char}(R/I) = 0$ for all $j$ in the same manner. Proceeding inductively, we have $a_ib_j = 0$ for all $i,j$.
    \end{itemize}
\end{proof}

By Proposition 3.12(2), one may conjecture that a ring $R$ is Armendariz of the Hurwitz series type when $R/I$ is Armendariz of the Hurwitz series type with $\text{char}(R/I) = 0$ for any nonzero proper ideal $I$ of $R$ which is Armendariz of the Hurwitz series type as a ring. However, the answer is negative as can be seen by the $2 \times 2$ upper triangular matrix ring over a reduced ring, say  $R = \begin{pmatrix}
    D & D \\
    0 & D
\end{pmatrix}$. Actually, $I = \begin{pmatrix}
    0 & D \\
    0 & 0
\end{pmatrix}$ and $R/I$ are Armendariz of the Hurwitz series type, but $R$ is non-Abelian, and so not Armendariz.

\end{document}